\documentclass[12pt]{amsart}
\usepackage{a4wide}
\usepackage{url}
\usepackage{hyperref}
\usepackage{amsrefs}
\usepackage{amsthm}
\usepackage{float}
\usepackage{amsmath, amssymb, mathtools}
\usepackage{graphicx}

\newtheorem{theorem}{Theorem}[section]
\newtheorem{proposition}[theorem]{Proposition}
\newtheorem*{theorem*}{Theorem}

\theoremstyle{definition}
\newtheorem{definition}[theorem]{Definition}

\theoremstyle{remark}
\newtheorem{remark}[theorem]{Remark}
\newtheorem{example}[theorem]{Example}

\numberwithin{equation}{section}

\def \CC {{\mathbb{C}}}
\def \RR {{\mathbb{R}}}
\def \ZZ {{\mathbb{Z}}}
\def \EE {{\mathbb{E}}}
\def \cA {{\mathcal{A}}}
\newcommand{\re}{\operatorname{Re}}
\newcommand{\im}{\operatorname{Im}}

\title{Singularities on maxfaces constructed by node-opening}
\author{Hao Chen}
\address{Institute of Mathematical Sciences, ShanghaiTech University, 201210 Shanghai, China.}
\email{chenhao5@shanghaitech.edu.cn}
\author{Anu Dhochak}
\address{Department of Mathematics, Shiv Nadar Institute of Eminence, Deemed to be University, Dadri 201314, Uttar Pradesh, India.}
\email{ad404@snu.edu.in}
\author{Pradip Kumar}
\address{Department of Mathematics, Shiv Nadar Institute of Eminence, Deemed to be University, Dadri 201314, Uttar Pradesh, India.}
\email{pradip.kumar@snu.edu.in, pmishra.math@gmail.com}
\author{Sai Rasmi Ranjan Mohanty}
\address{Department of Mathematics, Shiv Nadar Institute of Eminence, Deemed to be University, Dadri 201314, Uttar Pradesh, India.}
\email{sm743@snu.edu.in}
\begin{document}

\maketitle
\begin{abstract}
  The node-opening technique, originally designed for constructing minimal
  surfaces, is adapted to construct a rich variety of new maxfaces of high
  genus that are embedded outside a compact set and have arbitrarily many
  catenoid or planar ends, thus removing the scarcity of examples of maxfaces.
  The surfaces look like spacelike planes connected by small necks.  Among the
  examples are maxfaces of the Costa--Hoffman--Meeks type.  Although very
  fruitful, the main challenge of this paper is not the construction itself,
  but the analysis of the positions and natures of singularities on these
  maxfaces.  More specifically, we conclude that the singular set form curves
  around the waists of the necks.  In generic and some symmetric cases, all but
  finitely many singularities are cuspidal edges, and the non-cuspidal
  singularities are swallowtails evenly distributed along the singular curves.
\end{abstract}

\section{Introduction}

Maximal surfaces are zero mean curvature immersions in the Lorentz-Minkowski
space $\EE_1^3$. These surfaces emerge as solutions to the variational problem
of locally maximizing the area among spacelike surfaces.  They share several
similarities with minimal surfaces in $\RR^3$.  For instance, both are critical
points of the area functional and both admit Weierstrass-Enneper
representations. However, while there are rich examples of complete minimal
surfaces, the only complete maximal immersion is the plane~\cite{umehara2006}.

It is then natural to allow singularities.  Following~\cite{imaizumi2008,
lopez2007, umehara2006}, etc., we adopt the term \emph{maximal map} for maximal
immersions with singularities. A maximal map is called a \emph{maxface} if its
singularities consist solely of points where the limiting tangent plane contains
a light-like vector~\cite{umehara2006}.  Umehara and Yamada also defined
completeness for maxfaces~\cite{umehara2006}. Complete non-planar maxfaces
always possess a compact singularity set.  At the singularities, a maxface
cannot be embedded, regardless of whether the rest of the surface is embedded.
Therefore, following \cite{fujimori2009, kim2006, umehara2006}, we adopt
embeddedness in wider sense as follows: 

\begin{definition}\label{def:embedded}
  A complete maxface is embedded in a wider sense if it is embedded outside of
  some compact subset. 
\end{definition}

Now that singularities are allowed, there are many examples of complete
maxfaces, such as the Lorentzian catenoid and Kim-Yang toroidal maxface
\cite{kim2006}. In 2006, Kim and Yang \cite{kim2006} constructed complete
maximal maps of genus $k\geq 1$.  When $k = 1$, it is a complete embedded (in a
wider sense) maxface known as Kim-Yang toroidal maxface.  When $k>1$, they are
not maxfaces.  In~\cite{fujimori2009}, the authors constructed a family of
complete maxfaces $f_k$ for $k \ge 1$ with two ends; but when $k>1$, these may
not be embedded (in a wider sense).  Moreover, in 2016, Fujimori, Mohamed, and
Pember~\cite{fujimori2016h} constructed maxfaces of any odd genus $g$ with two
complete ends (if $g=1$, the ends are embedded) and maxfaces of genus $g=1$ and
three complete embedded ends.  In 2024, the third author, along with Bardhan
and Biswas, proved the existence of higher genus maxfaces with one Enneper end
in \cite{BBP}.

To the best of our knowledge, all higher-genus maxfaces in the literature have
only two or three ends, usually of the catenoid type. Very few known
higher-genus maxfaces are embedded (in a wider sense). Open problems
in~\cite{fujimori2009} express the hope for a large collection of examples of
complete maxfaces that are embedded (in a wider sense), with higher genus and
many ends.

The scarcity of examples is surprising.  After all, the minimal surfaces and
maxfaces admit similar Weierstrass-Enneper representations.  But to construct
higher genus embedded (in a wider sense) maxfaces with the required type of
ends, it is usually not as direct as simply manipulating the Weierstrass data.
One challenge is the singularities:  While proposing the Weierstrass data, we
must ensure that the singular curve does not approach the ends and is compact.
Another challenge is the period problem, as illustrated in the following
example.

\begin{example}[See \cite{BBP}*{Section 1.1}]
	Consider the Costa minimal surface:
	\[
		M = \{(z, w) \in \mathbb{C}\times\mathbb{C} \cup \{(\infty, \infty)\} \mid w^2 = z(z^2-1)\} \setminus \{(0, 0), (\pm1, 0)\}
	\]
	with the data $\{\frac{a}{w}, \frac{2a}{z^2-1}\mathrm{d}z\}$, where $a \in
	\mathbb{R}^+ \setminus \{0\}$.  If there exists a companion maxface of the
	Costa surface, then its data should be $\{-\frac{i a}{w}, \frac{2i
	a}{z^2-1}\mathrm{d}z\}$ defined on $M$. Let $\tau$ be a one-sheeted loop
	around $(-1, 0)$ that does not contain $(1, 0)$. Then, \(\int_{\tau}
	\left(\frac{2i a}{z^2-1} \mathrm{d} z\right) = 2i a\), where \(a \neq 0\).
	Therefore, the period problem is not solved for the corresponding maxface. So
	the Costa surface does not have a companion maxface.  
\end{example}

\medskip

In this paper, we adapt the node-opening technique to construct a rich variety
of complete maxfaces of high genus that are embedded (in a wider sense) and
have an arbitrary number of spacelike ends, thus removing the scarcity of
examples.

The node-opening technique is a Weierstrass gluing method developed by
Traizet~\cite{traizet2002e}.  It constructs a family of surfaces depending on a
real parameter $t$.  The approach starts at $t=0$ with Weierstrass data defined
on a Riemann surface with nodes, then ``deforms'' to Riemann surfaces for $t>0$
by opening the nodes into necks and, at the same time, ``deforms'' the
corresponding Weierstrass data  using the Implicit Function Theorem.  The
node-opening technique has been very successful in constructing a rich variety
of minimal surfaces~\cite{traizet2002e}.  To the best of our knowledge, the
current paper marks the first application of the technique to surfaces in the
Lorentz-Minkowski space.

The Weierstrass gluing method has several advantages over other methods.  On
the one hand, in the existing literature on maxfaces, authors often need to
assume symmetries to make the construction possible, hence only produce
examples restricted to symmetries.  The gluing technique has been a very
powerful tool to break symmetries in minimal surfaces; in some sense, the
technique was developed for this purpose~\cite{traizet2002e}.  We will see
later that it is equally powerful in breaking symmetries for maximal surfaces,
hence ideal for removing the scarcity of examples. On the other hand, while the
PDE gluing method is also popular for constructing minimal surfaces, the
existence of singularities makes it difficult to be adapted for maxfaces.  More
specifically, one needs to identify (glue) two curves in the process, but the
analysis would be difficult if the curves contain singularities.  In the
Weierstrass gluing process, we instead identify (glue) two annuli; hence, we
can bypass the singularities.

The node-opening construction for maxfaces turns out to be very similar to that
for minimal surfaces, so we will only provide a sketch.  We will first give a
Weierstrass data in Section~\ref{sec:WeierstrassData}, leaving many parameters
to be determined later in Appendix~\ref{sec:ift} by solving the divisor problem
and the period problem using the Implicit Function Theorem.

This similarity implies that, for any minimal surface constructed by opening
nodes, there is a corresponding maxface also constructed by opening nodes.
This correspondence between maximal and minimal immersions is different from
the usual correspondence through Weierstrass data~\cite{umehara2006}.  In
Section~\ref{sec:example}, by simply comparing notes, we obtain a rich variety
of new maxfaces with high genus and arbitrarily many space-like ends, thereby
remove the scarcity of examples of maxfaces.  Among the examples are the
Lorentzian Costa and Costa--Hoffman--Meeks (CHM) surfaces and their
generalizations with arbitrarily many ends, providing positive answers to the
open problems in~\cite{fujimori2009}.  To the best of our knowledges, this is
the first time that Lorentzian analogues of CHM surfaces were constructed.

\begin{remark}\label{rem:periodic}
	One could also use the node-opening technique to glue catenoids into periodic
	maxfaces, even of infinite genus, just by mimicking~\cite{traizet2002r,
	traizet2008t, morabito2012, chen2021, chen2023}. We believe that many examples
	in the existing literature, e.g. Lorentzian Riemann examples and Schwarz P
	surfaces, etc, can be produced in this way~\cite{fujimori2009, lopez2000m}.
	However, we do not plan to implement such constructions.
\end{remark}

\medskip

We see the the node-opening construction itself is very fruitful, but
technically not that exciting.  In the current paper, most effort is devoted to
the more challenging task of analyzing the singularities on the constructed
maxfaces.

Complete non-planar maxfaces always appear with singularities, such as cuspidal
edges, swallowtails, cuspidal crosscaps, and cone-like singularities, to name a
few. We refer readers to~\cite{umehara2006, sai2022, kumar2020} to explore
various singularities on maxfaces.  The nature of singularities can be told
from the Weierstrass data.  However, for the maxfaces we construct, the
Weierstrass data is not given explicitly; rather, its existence is implied
using the Implicit Function Theorem.  Hence the analysis of the singularities
must also be performed implicitly, which is a challenging task.  Nevertheless,
we managed to perform the analysis in various situations.

To prepare for the analysis, we first recall in
Section~\ref{sec:Governingfunction} the governing functions whose derivatives
determine the nature of singularities.  Then in
Section~\ref{sec:Partialderivativeof A}, we perform an elaborate calculation of
the higher order derivatives of the governing function.  It can be seen as
generalizing Traizet's calculation~\cite{traizet2008t} of the first derivative
of the height differential, hence might have other use in similar node~opening
constructions.  These allow us to conclude various results about the nature of
singularities in Section~\ref{sec:natureofSIngularties}, which we summarize
below.

Components of the singular set are waists around the necks. In
Theorem~\ref{thm:noncuspidalvariety}, we prove that around a specific neck and
for sufficiently small, non-zero $t$, either the singular set is mapped to a
single point (cone-like singularity), or all but a finite number of singular
points are cuspidal edges. Moreover, the finitely many non-cuspidal
singularities are generalized $A_k$ singularities, their positions on the waist
depend analytically on $t$, and their types do not vary for sufficiently small
$t$. Then, in Proposition~\ref{prop:swallowtails}, we show that, generically,
there are four swallowtail singularities around a neck. The non-generic cases
are generally hard to study, but in Section~\ref{sec:symmetries}, we managed to
analyze the singularities in the presence of rotational and reflectional
symmetries. In Section~\ref{sec:example}, we will use the Lorentzian Costa and
Costa--Hoffman--Meeks surfaces to exemplify our results on singularities.

\medskip 

\subsection*{Acknowledgment}
The first and third authors would like to extend their sincere gratitude to
Professor S.\ D.\ Yang for graciously inviting them to The 3rd Conference on
Surfaces, Analysis, and Numerics at Korea University. They are truly grateful
for the opportunity, and our current work has commenced.

\section{Main results}\label{sec:mainresult}

\subsection{Node-opening construction}

We want to construct maxfaces that look like horizontal (spacelike) planes
connected by small necks.  For that, we consider $L$ horizontal planes, labeled
by integers $l \in [1, L]$.  We want $n_l > 0$ necks at level $l$, that is,
between the planes $l$ and $l+1$, $1 \le l < L$.  For convenience, we adopt the
convention that $n_0 = n_L = 0$, and write $N = \sum n_l$ for the total number
of necks.  Each neck is then labeled by a pair $(l,k)$ with $1 \le l < L$ and $1
\le k \le n_l$.

To each plane is associated a real number $Q_l$, indicating the logarithmic
growth of the catenoid ($Q_l \ne 0$) or planar ends ($Q_l = 0$).  To each neck
is associated a complex number $p_{l,k} \in \CC$ indicating its
horizontal limit position at $t=0$.  We write $p = (p_{l,k})_{1 \le l < L, 1
\le k \le n_l}$ and $Q=(Q_l)_{1 \le l < L}$.  The pair $(p, Q)$ is called a
\emph{configuration}.

Given a configuration $(p, Q)$, let $c_l$ be the real numbers that solve
\[
  Q_l = n_{l-1}c_{l-1} - n_l c_l, \qquad 1 \le l \le L,
\]
under the convention that $c_0 = c_L = 0$.  A summation over $l$ yields that
$\sum Q_l=0$, which is necessary for $c = (c_l)_{1 \le l \le L}$ to be uniquely
determined as a linear function of $Q$.  In fact, we may even replace $Q$ by
$c$ in the definition of a configuration.  Geometrically, $c_l$ corresponds to
the ``size'' of the necks at level $l$.  

\medskip

For the neck $(l,k)$ in a configuration, we define the force $F_{l,k}$ on the neck as 
\[
	F_{l,k}=
	\sum_{1 \le i \ne k \le n_l}\frac{2c_l^2}{p_{l,k}-p_{l,i}}-
	\sum_{i=1}^{n_{l+1}}\frac{c_lc_{l+1}}{p_{l,k}-p_{l+1,i}}-
	\sum_{i=1}^{n_{l-1}}\frac{c_l c_{l-1}}{p_{l,k}-p_{l-1,i}}.
\]
Note that we have necessarily $\sum F_{l,k} = 0$.

Alternatively, let
\[
  {\omega}_l =
  - \sum_{k = 1}^{n_l} \frac{c_l\; dz}{z-p_{l,k}}
  + \sum_{i = 1}^{n_{l-1}} \frac{c_{l-1}\;  dz}{z-p_{l-1,k}}
\]
be the unique meromorphic $1$-form on $\CC_l$ with simple poles at $p_{l,k}$
and $p_{l-1,k}$, respectively with residues $-c_l$ and $c_{l-1}$.  Then, the
force is given by 
\[
  F_{l, k} = \frac{1}{2}\operatorname{Res}_{p_{l,k}}\left(
    \frac{\omega_l^2}{dz} + \frac{\omega_{l+1}^2}{dz}
  \right).
\]

\begin{definition}
	A configuration is \emph{balanced} if $F_{l,k}=0$ for all $1 \le l < L$ and
	$1 \le k \le n_l$, and is \emph{rigid} if the differential of $F =
	(F_{l,k})_{1 \le l < L, 1 \le k \le n_l}$ with respect to $p$ has a complex
	rank of $n-2$.
\end{definition}

In fact, $n-2$ is the maximum possible rank.  To see this, note that the forces
$F$ are invariant under the translations and complex scalings of $p$.

A necessary condition for the balance is
\[
	W:=\sum_{l=1}^L \sum_{k=1}^{n_l} p_{l,k} F_{l,k}
	= \sum_{l=1}^{L-1}n_l(n_l-1)c_l^2-\sum_{l=1}^{L-2}n_ln_{l+1}c_lc_{l+1} = 0.
\]

We now state our first main result.

\begin{theorem}\label{thm:main}
	Let $(p,Q)$ be a balanced and rigid configuration such that the differential
	of $Q \mapsto W$ has rank $1$.  Then, for sufficiently small $t$, there is a
	smooth family $M_t$ of complete maxfaces with the following asymptotic
	behaviors as $t \to 0$
	\begin{itemize}
    \item The maxfaces are of genus $N-L+1$ with $L$ space like ends, whose
    	logarithmic growths converge to $Q_l$.

    \item After suitable scalings, the necks at level $l$ converge to
    	Lorentzian catenoids;

    \item $M_t$ scaled by $t$ converges to an $L$-sheeted space-like plane with
    	singular points at $p_{l,k}$.
	\end{itemize}
 	Moreover, $M_t$ is embedded in a wider sense for sufficiently small $t$ if
 	$Q_1 < Q_2 < \cdots < Q_L$. 
\end{theorem}

The balance and non-degeneracy conditions for maxfaces turn out to be exactly
the same to those for minimal surfaces, and many balanced configurations have
been found when constructing minimal surfaces.  So we obtain a rich variety of
new maxfaces simply by comparing notes.  Some of the examples are listed in
Section~\ref{sec:example}.  Thereby we remove the scarcity of examples of
maxfaces.

Although very fruitful, the construction (proof of Theorem \ref{thm:main}) is straightforward.
In particular, the proof is very similar to that for minimal
surfaces~\cite{traizet2002e}, with only slight modifications. So we will only
provide an sktech. In Section~\ref{sec:WeierstrassData}, we will give a
Weierstrass data with undetermined parameters.  In Appendix~\ref{sec:ift}, we
will sketch the use of Implicit Function Theorem to find parameters that solve
the divisor problem and the period problem.

\subsection{Singularities}

Let us first define
\begin{equation}\label{eqn:R}
  R^{(r)}_{l,k}(\theta) = \begin{dcases} 
    \im \bigg(
      e^{(r+1) i \theta} 
      \overline{\operatorname{Res}_{p_{l,k}}\frac{\omega_l^{r+2}}{dz}}
      - e^{-(r+1) i \theta}
      \operatorname{Res}_{p_{l,k}}\frac{\omega_{l+1}^{r+2}}{dz}
    \bigg), & l \text{ odd},\\
    \im \bigg(
      e^{(r+1) i \theta} 
      \operatorname{Res}_{p_{l,k}}\frac{\omega_l^{r+2}}{dz}
      - e^{-(r+1) i \theta}
      \overline{\operatorname{Res}_{p_{l,k}}\frac{\omega_{l+1}^{r+2}}{dz}}
    \bigg), & l \text{ even}.
  \end{dcases}
\end{equation}
Our main result about singularities are summarized below.

\begin{theorem}
	On a maxfaces constructed above, for sufficiently small non-zero $t$,
  \begin{itemize}
    \item The singular set has $N$ singular components, each being a curve
    	around the waist of a neck.

    \item The singularities are all nondegenerate.

    \item If the singular curve is not mapped to a single point (cone-like
    	singularity), then all but finitely many singular points are cuspidal
    	edges.

    \item The non-cuspidal singularities are generalized $A_k$ singularities,
    	their positions vary analytically with $t$, and their types do not vary.

    \item If $R^{(1)}_{l,k} \ne 0$, then there are exactly four non-cuspidal
    	singularities around the neck $(l,k)$, and they are all swallowtails.
    	Moreover, they tend to be evenly distributed on the waist as $t \to 0$.
  \end{itemize}
\end{theorem}

The results above cover the generic situations.  The non-generic cases are hard
to analyze.  However, if the configuration has symmetries, we have the
following results:

\begin{enumerate}
  \item Assume that the configuration has rotational symmetry of order $r>1$
   	around a neck and $R^{(r-1)}_{l,k} \ne 0$, then for sufficiently small
   	non-zero $t$, there are $2r$ swallowtail singularities around the neck, and
   	they tend to be evenly distributed as $t \to 0$.

  \item Assume that the configuration has a vertical reflection plane that cuts
  	through a neck, then the singularity around the neck that is fixed by the
  	reflection is non-cuspidal.

  \item Assume that the configuration of necks has a horizontal reflection plane
  	that cuts through a neck, then the singular curve around the neck is mapped to
  	a conelike singularity.
\end{enumerate}

We will demonstrate these situations by examples in the next section.

\begin{remark}
  Because the singularities are all non-degenerate for sufficiently small $t$,
  we do not have any cuspidal cross caps.  However, if we glue Lorentzian
  helicoids into maxfaces (see~\cite{traizet2005, freese2022, chen2022} for
  constructions of minimal surfaces), then by the duality between swallowtails
  and cuspidal cross caps~\cite{fujimori2008},  we would expect no
  swallowtails but only cuspidal cross caps.
\end{remark}

\section{Examples}\label{sec:example}

\subsection{Configurations from minimal surfaces}

Note that the balance and non-degeneracy conditions are exactly the same for
maxfaces and for minimal surfaces.  So all the configurations found
in~\cite{traizet2002e} that give rise to the minimal surface also give rise to
maxfaces.  We now summarize some balanced and non-degenerate configurations (or
methods to produce configurations) from~\cite{traizet2002e}, and the
corresponding maxfaces.

\begin{itemize}
  \item The simplest configuration would have a single neck, given by
    \begin{gather*}
      L=2,\quad n_1=1,\quad p_{1,1}=0,\\
      Q_1 = -1,\quad Q_2 = 1,\quad (\text{so } c_1 = 1).
    \end{gather*}
    The corresponding maxface is the Lorentzian catenoid.  It is of genus $0$,
    has two spacelike ends.

  \item The Costa--Hoffman--Meeks (CHM) configurations are given by
    \begin{gather*}
      L=3, \quad n_1 = 1,\quad n_2 = m,\\
      p_{1,1} = 0,\quad p_{2, m}=e^{2k\pi i/m},\quad 1 \le k \le m,\\
      Q_1=1-m,\quad Q_2=-1,\quad Q_3=m\quad (\text{so } c_1=m-1,\quad c_2=1).
    \end{gather*}
    We call the corresponding maxfaces Lorentzian Costa ($m=2$) or
    Costa--Hoffman--Meeks (CHM) surfaces ($m>2$). They provide positive answers
    to Problem 1 in~\cite{fujimori2009}.
    \begin{theorem}
      For each $g > 1$, there exist complete embedded (in a wider sense)
      maxfaces with $3$ spacelike ends and genus $g$.
    \end{theorem}

  \item Dihedral configurations with arbitrary number of ends were explicitly
  	constructed in~\cite{traizet2008n}.  They are given by $n_1 = 1$ and $n_2 =
  	\ldots n_{L-1} = m$, subject to a dihedral symmetry of order $m$.  These
  	configurations are balanced, and non-degenerate for a generic choice of
  	$c_l$.  The embeddedness condition $Q_1 < \cdots < Q_L$ is satisfied if $m
  	> 2(L-2)$.  Taking $m = 2L-3$, we obtain generalizations of CHM maxfaces
  	that provide positive answer to Problem 2 in~\cite{fujimori2009}.
    \begin{theorem}
      For each $L > 3$, there exist complete embedded (in a wider sense)
      maxfaces with $L$ spacelike ends and genus $2(L-2)^2$.
    \end{theorem}

  \item Numerical examples can be obtained by the polynomial method.  More
   	specifically, let
    \[
     	P_l = \prod_{k=1}^{n_l} (z-p_{l,k}), \quad P = \prod_{l=1}^{L-1} P_l,
    \]
    then the configuration is balanced if
    \[
      \sum_{l=1}^{L-1} c_l^2 P \frac{P''_l}{P_l} - \sum_{l=1}^{L-2} c_l c_{l+1} P \frac{P'_l P'_{l+1}}{P_l P_{l+1}} \equiv 0.
    \]
    If a polynomial solution to this differential equation has only simple
    roots, then the roots correspond to the positions of nodes (up to
    permutations).

  \item Implicit examples can be obtained by perturbing ``singular''
   	configurations.

    More specifically, consider a partition $I_1, ..., I_m$ of the nodes and a
    family of configurations given by $p_{l,k}^\lambda = \hat p_\mu +
    \lambda_{\mu} \tilde p_{l,k,\mu}$ when $(l,k) \in I_\mu$.  Then, the limit
    configuration $p^0$ is singular.  A force can be defined for the limit
    configuration in terms of $\hat p$ and the partition.  For each $\mu$,
    $\tilde p_{l,k,\mu}$ form a subconfiguration $\tilde p_\mu$.

    In the backward direction, Traizet~\cite{traizet2002e} found
    sufficient conditions to recover configuration $p^\lambda$ from the limit
    configuration $\hat p$ and the sub-configurations $\tilde p_\mu$.  In
    particular, the limit configuration and all sub-configurations should be
    balanced.  This result was used to construct examples with no symmetry.
    Using exactly the same configuration, we also obtain
    \begin{theorem}
      There exist complete embedded (in a wider sense) maxfaces with no
      nontrivial symmetries.
    \end{theorem}
\end{itemize}

\begin{remark}
  As we have noticed in Remark~\ref{rem:periodic}, a similar technique can
  produce periodic maxfaces.  Although we do not plan to implement such
  constructions, it is predictable that the balance and non-degeneracy
  conditions are again the same for maxfaces and for minimal surfaces.  Hence,
  the periodic configurations in~\cite{traizet2002r, traizet2008t} and even the
  nonperiodic infinite-genus configurations in~\cite{morabito2012, chen2021}
  should also give rise to maxfaces. 
\end{remark}

\subsection{Singularities on Lorentzian Costa and CHM surfaces}

The Lorentzian Costa and CHM surfaces are particularly interesting in regard to
singularities.

The Lorentzian Costa surface has three disjoint singular curves in the waist of
each neck.  To analyse its singularities, we need to compute
${R}_{l,k}^{(r)}(\theta)$ as in Equation \ref{eqn:R}. We have 

\begin{align*} 
 	R_{1,1}^{(1)}(\theta) &= 6 \sin \theta, \\
 	R_{2,1}^{(1)}(\theta) &= -3 \sin \theta, \\
 	R_{2,2}^{(1)}(\theta) &= -3 \sin \theta.
\end{align*}

Therefore, by Theorem \ref{thm:noncuspidalvariety} and Proposition
\ref{prop:swallowtails}, we can conclude that all non-cuspidal
singularities are swallowtails for sufficiently small $t$.

The computation above did not rely on symmetries.  For a Lorentzian CHM
surface, by Propositions~\ref{prop:rotations} and~\ref{prop:verticalref}, we
can already conclude from its dihedral symmetry that, for sufficiently small
non-zero $t$, there are $2m$ non-zero in the waist of the neck $(1,1)$, all are
swallowtails and are fixed by the vertical reflections.  See
Figure~\ref{fig:CHM}.

\begin{figure}[h!]
  \includegraphics[width=0.6\textwidth]{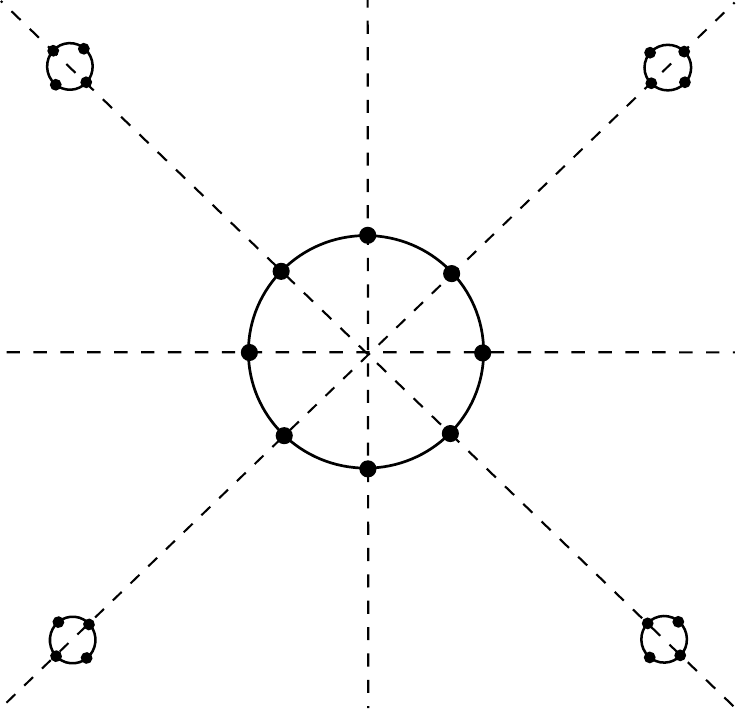}
  \caption{
    Sketch of singularity structure of a CHM surface with $m=4$.  The dashed
    lines indicate the reflection symmetries.  The solid curves are singular
    curves in the waist of the necks.  The singularities are cuspidal edges
    except at the dots, where the singularities are swallowtails.
    \label{fig:CHM}
  }
\end{figure}

Alternatively, we could also perform an explicit computation that
$R_{1,1}^{(r)} = 0$ for all $1 \le r \le m-2$ while
\[
 	R_{1,1}^{(m-1)} = (m+1)m(m-1)^m \sin (m\theta) \not\equiv 0.
\]
Moreover, for $1 \le k \le m$, we have
\[
 	R_{2,k}^{(1)} = (1-m^2) \sin(2\theta - 4 k \pi /m) \not\equiv 0.
\]

Then, by Theorem \ref{thm:noncuspidalvariety} and Proposition
\ref{prop:swallowtails}, we can conclude that for sufficiently small
non-zero $t$, there are $2m$ non-cuspidal singularities in the waist of the
center neck, and four non-cuspidal singularities in the waist of the other
necks, and they are all swallowtails.

In Figure~\ref{fig:CHM3D}, we show the numerical pictures of Lorentzian CHM
surfaces with $m=4$ and $m=5$, and zoom in to show the details of singularities
around the center neck.

\begin{figure}[h!]
  \includegraphics[width=0.5\textwidth]{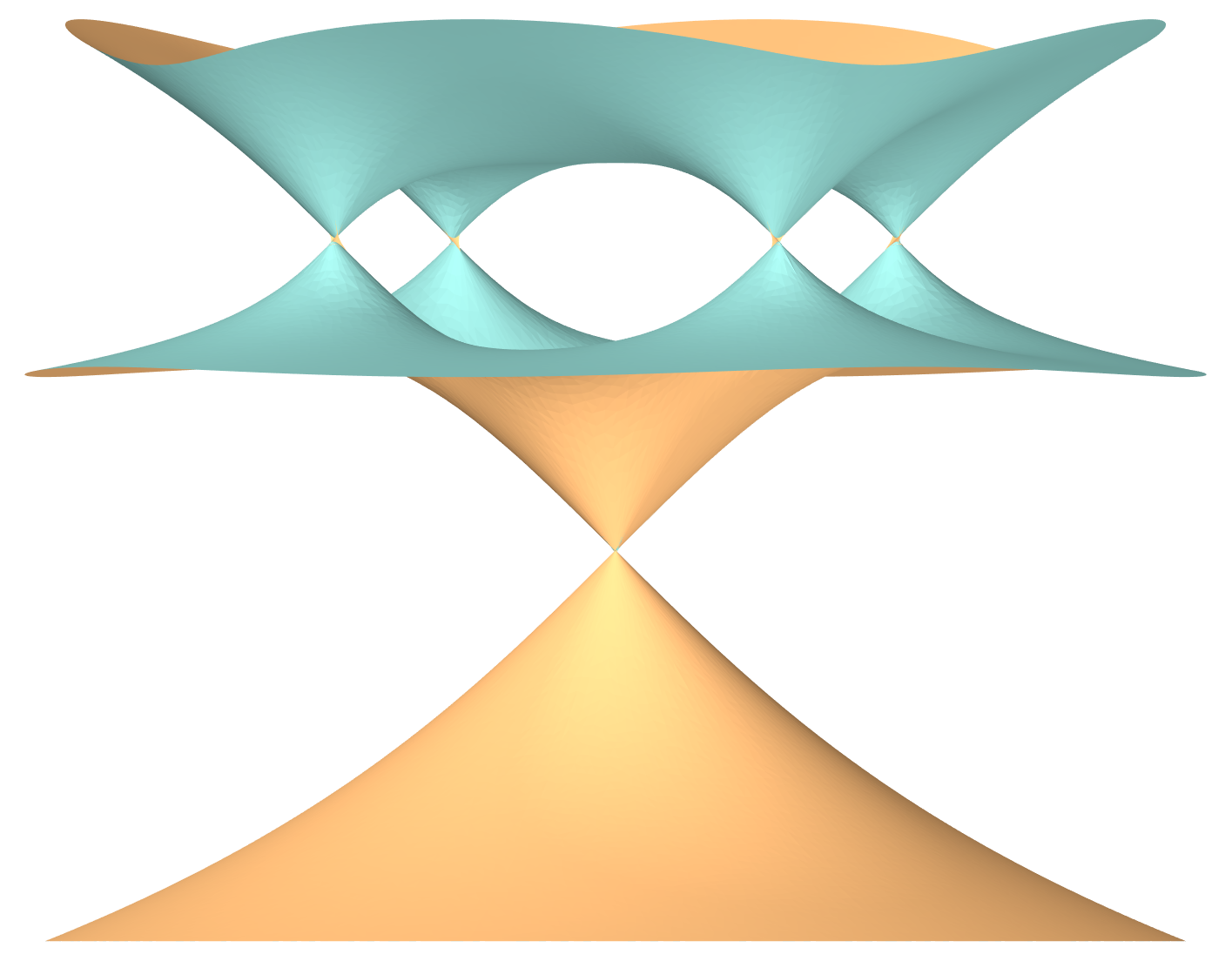}
  \includegraphics[width=0.3\textwidth]{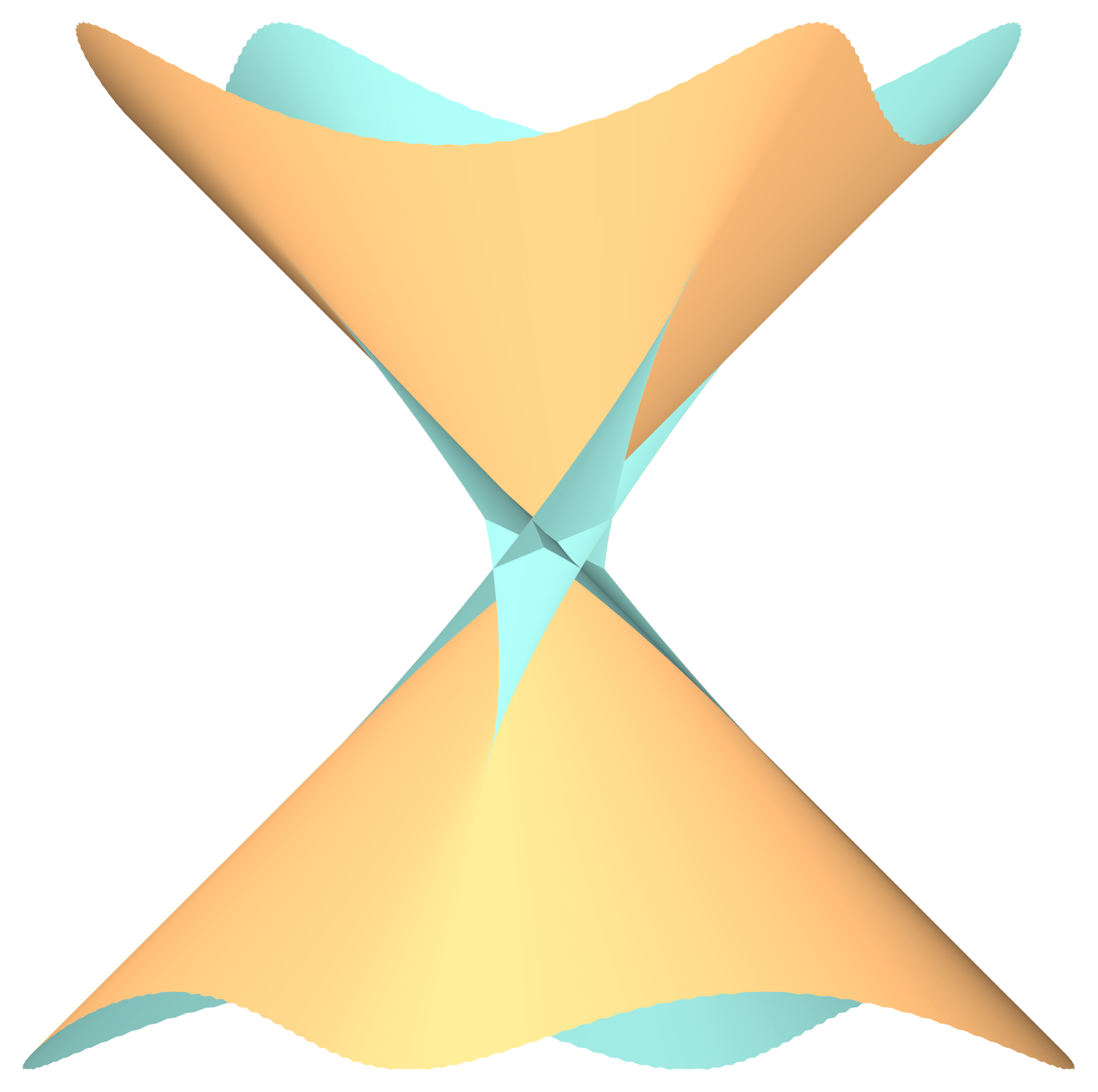}
  \includegraphics[width=0.5\textwidth]{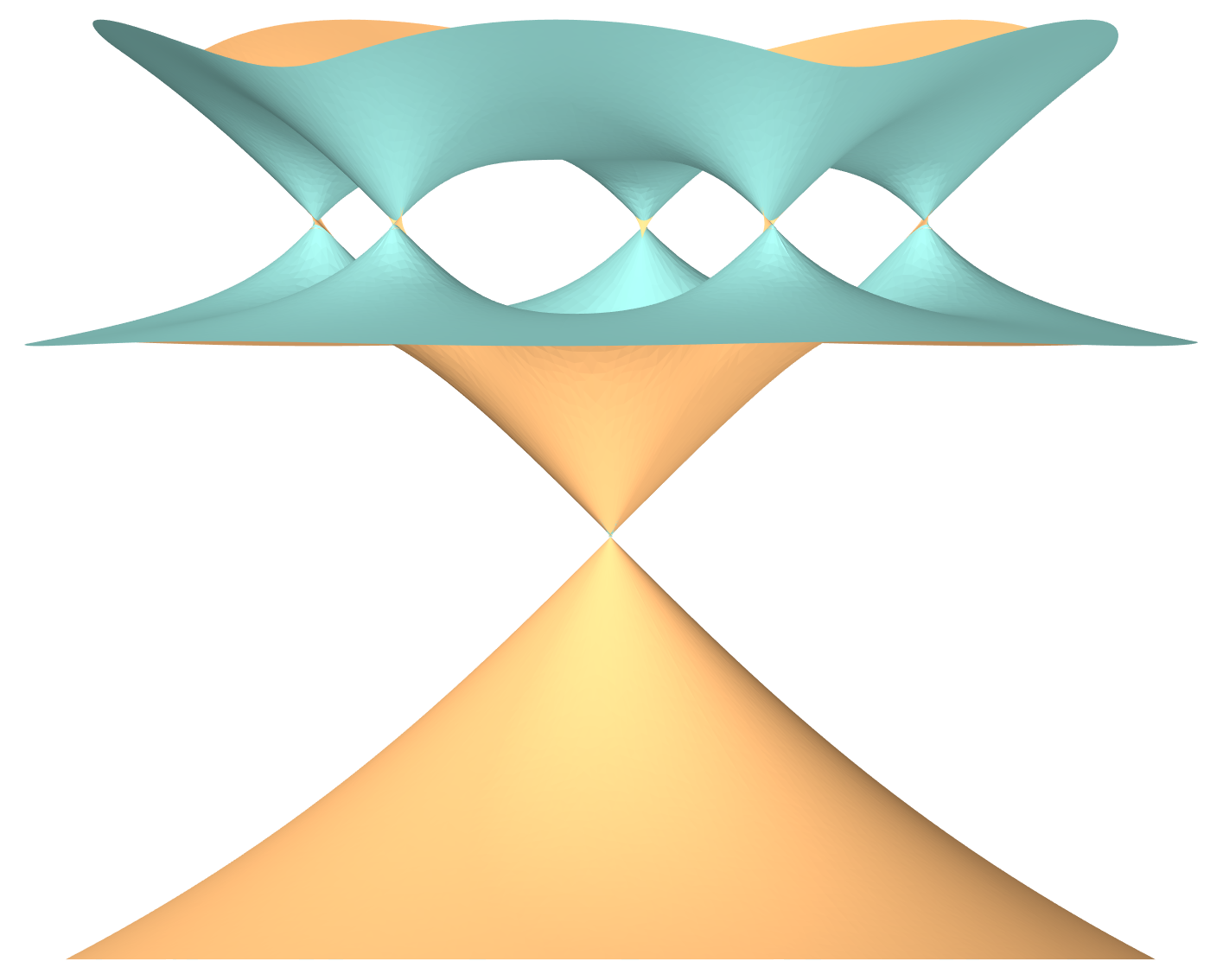}
  \includegraphics[width=0.3\textwidth]{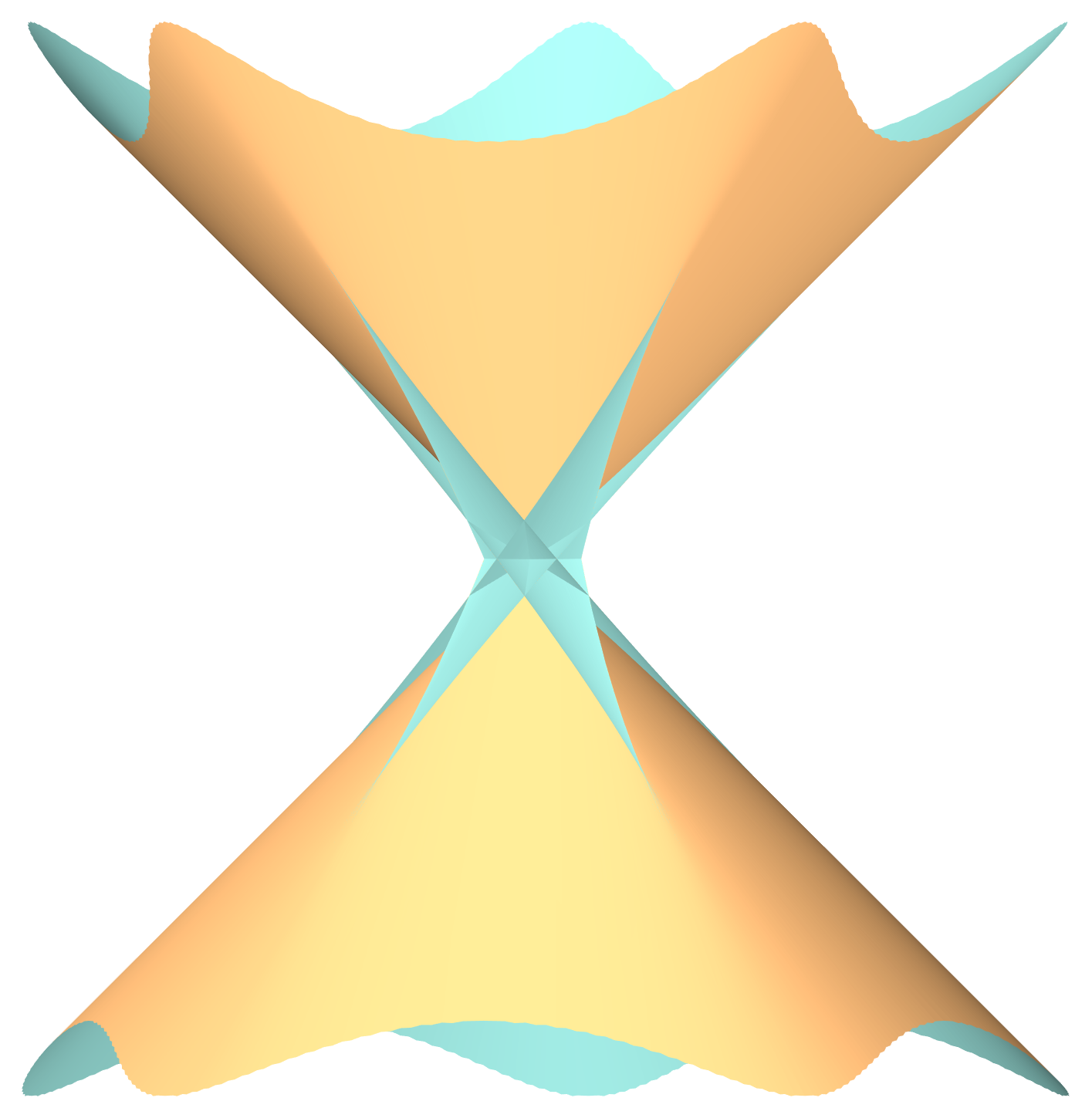}
  \caption{
    CHM sufaces with $m=4$ (top) and $m=5$ (bottom).  On the right-hand side
    are the zoom-ins of the center neck, showing details of the cuspidal edges
    and swallowtails.  To make this picture, we use the node-opening
    construction as in~\cite{traizet2008t}, which is different but equivalent
    to the construction in the current paper, and better suited for numerical
    computations~\cite{traizet2008n}.\label{fig:CHM3D}
  }
\end{figure}

\section{Weierstrass data}\label{sec:WeierstrassData}

We construct maxfaces using a Weierstrass--Enneper-like parameterization,
namely
\begin{equation}\label{eq:maxWeierstrass}
  M \ni z \mapsto \re \int^z \Big(  \frac{1}{2}(g^{-1} + g), \frac{i}{2}(g^{-1} - g), 1 \Big) dh \in \EE^3_1,
\end{equation}
where $M$ is a Riemann surface, possibly with punctures corresponding to the
ends, $g$ is a meromorphic function, and $dh$ a holomorphic $1$-form on $M$,
subject to the following conditions:
\begin{description}
  \item[Divisor condition] Away from the punctures, we must have
    \[
      (g)_0-(g)_\infty= (dh)_0
    \]
    for the Weierstrass integrands to be holomorphic.  The behavior at the
    punctures depend on the type of the ends.

  \item[Period condition] For all closed curves $\gamma$ on $M$, we have
    \begin{align}
      \overline{\int_{\gamma}g^{-1}dh} + \int_{\gamma}gdh &= 0, \label{eq:hperiod}\\
      \re\int_{\gamma}dh &= 0. \label{eq:vperiod}
    \end{align}
    So, closed curves in $M$ are mapped to closed curves on the surface.  This
    guarantees that the immersion is well-defined.

  \item[Regularity condition] $|g|$ is not identically $1$.  In fact, the
   	pullback metric on the Riemann surface $M$ is given by
   	$ds^2=\frac{1}{4}\left(|g|^{-1}-|g|\right)^2|dh|^2$.  In view of the
   	divisor condition, the singularity set for maxfaces is then given by
   	$\{p\in M: |g(p)|=1 \}$.  The regularity condition guarantees that the
   	immersion is regular.
\end{description}

\begin{remark}
  For minimal surface, the horizontal period condition~\eqref{eq:hperiod} would
  have a minus sign in the middle, and the pull-back metric would have a plus
  sign.
\end{remark}

We propose the Weierstrass data in the following, as they will be useful later
for the analysis of singularities.  Some parameters are left undetermined.
They will be determined later in Appendix~\ref{sec:ift} using the Implicit
Function Theorem.

\subsection{The Riemann Surface}\label{sec:riemannsurface}

We construct the Riemann surface by node-opening as follows:

To each of the $L$ horizontal (spacelike) planes is associated a copy of the
complex plane $\CC$, which can be seen as the Riemann sphere with punctures at
$\infty$.  The $L$ copies of $\CC$, as well as their punctures, are then
indexed by $l$, $1 \le l \le L$.  To each neck at level $l$, $1 \le l < L$ is
associated a puncture $a_{l,k} \in \CC_l$ and a puncture $b_{l,k} \in
\CC_{l+1}$, $1 \le k \le n_l$.

Initially at $t=0$, we simply identify $a_{l,k}$ with $b_{l,k}$ for all $1 \le
l < L$ and $1 \le k \le n_l$ to obtain a noded Riemann surface $\Sigma_0$.  As
$t$ increases, fix local coordinates $v_{l,k}$ in the neighborhood of
$a_{l,k}$, and local coordinates $w_{l,k}$ in the neighborhood of $b_{l,k}$; a
concrete choice will be made soon later.  We may fix an $\varepsilon$
sufficiently small and independent of $l$ and $k$ so that the disks
$|v_{l,k}|<2\varepsilon$ and $|w_{l,k}|<2\varepsilon$ are all disjoint.  For $t
< \varepsilon$, we may remove the disks $|v_{l,k}| < t^2/\varepsilon$ and
$|w_{l,k}| < t^2/\varepsilon$, and identify the annuli
\[
  t^2/\varepsilon < |v_{l,k}| < \varepsilon \qquad t^2/\varepsilon < |w_{l,k}| < \varepsilon
\]
by
\[
  v_{l,k} w_{l,k} = t^2.
\]
The resulting Riemann surface is denoted by $\Sigma_t$.

\subsection{Gauss map and local coordinates}\label{sec:Gaussmap}

We define on $\CC_l$ the meromorphic function
\[
 	g_l:=\sum_{k=1}^{n_l}\frac{\alpha_{l,k}}{z-a_{l,k}}+\sum_{k=1}^{n_{l-1}}\frac{\beta_{l-1,k}}{z-b_{l-1,k}}.
\]
Then the Gauss map $g$ is defined on $\Sigma_t$ as
\[
 	g(z) =
 	\begin{cases}
    t g_l(z) & \text{if } z\in\CC_l \text{ and $l$ is odd,}\\
    1/(t g_l(z)) & \text{if } z\in\CC_l\text{ and $l$ is even.}
 	\end{cases}
\]
As $t \to 0$, the Gauss map converges to that of catenoids around the necks.
Note that $1/g_l$ provides local coordinates $v_{l,k}$ around $a_{l,k}$ and
$w_{l-1,k}$ around $b_{l-1,k}$.  From now on, we adopt these local coordinates
for the construction of $\Sigma_t$.

\subsection{The height differential}\label{sec:dh}

Recall the period conditions $\re\int_\gamma dh = 0$ for every closed cycles
$\gamma$ of $\Sigma_t$.  Define
\[
 	\Omega_l:=\{z \in \CC_l: |v_{l,i}|\geq \varepsilon\quad\forall 1 \le i \le n_l \quad\text{and}\quad 
 	|w_{l-1,j}|\geq\varepsilon \quad\forall 1 \le j \le n_{l-1}\},
\]
where $\varepsilon$ was previously fixed for the construction of $\Sigma_t$.
Let $\gamma_{l,k}$ be small clockwise circles in $\Omega_l$ around $a_{l,k}$;
they are homologous to counterclockwise circles in $\Omega_{l+1}$ around $b_{l,k}$.
We close the vertical periods by requiring that $\int_{\gamma_{l,k}} dh = 2\pi
i r_{l,k}$ for real numbers $r_{l,k}$.  Moreover, as we expect catenoid ends at
$\infty_l$, we require that the height differential $dh$ has simple poles of
residues $-R_l \in \RR$ at $\infty_l$, $1 \le l \le L$.  By the Residue
Theorem, it is necessary that $\sum R_l = 0$ and
\[
  \sum_{k=1}^{n_l} r_{l,k} - \sum_{k=1}^{n_{l-1}} r_{l-1,k} = -R_l.
\]
So, it suffices to prescribe the residue $-R_l$ and the periods around
$\gamma_{l,k}$ for $1 < k \le n_l$. 

By~\cite{traizet2002e}, these requirements uniquely determine the
height differential $dh$.  Moreover, as $t \to 0$, $dh$ converges uniformly on
a compact set of $\Omega_l$ to the form
\begin{equation}\label{eq:heightdiff}
 \sum_{k = 1}^{n_l} \frac{-r_{l,k} dz}{z-a_{l,k}} + \sum_{k = 1}^{n_{l-1}} \frac{r_{l-1,k} dz}{z-b_{l-1,k}} 
\end{equation}

We want catenoid or planar ends at the punctures $\infty_l$.  This translates
to the following divisor condition at $\infty_l$:  Whenever $g$ has a simple
zero or pole there, $dh$ must have a simple zero; this corresponds to the
catenoid ends.  On the other hand, whenever $g$ has a zero or pole of
multiplicity $m > 1$ at $\infty_l$, $dh$ must have a zero or multiplicity
$m-2$; this corresponds to the planar ends.  Because $dz$ has a pole of order
$2$ at the punctures $\infty$, our divisor condition can be formulated as

\begin{equation} 
  (g)_0-(g)_\infty= (dh/dz)_0.
\end{equation}

\medskip

The next step of the construction is to determine the parameters using the
Implicit Function Theorem.  More specifically, with all parameters varying in a
neighborhood of their initial values (at $t=0$), we need to prove the following
\begin{itemize}
	\item There exist unique values for $\alpha$ and $\beta$, depending
		analytically on other parameters such that the divisor conditions are
		satisfied.

	\item With $\alpha$ and $\beta$ given above, there exist unique values for
		$r$, depending analytically on remaining parameters, such that the vertical
		period condition are satisfied.

	\item With $\alpha$, $\beta$ and $r$ given above, there exist unique values
		for $a$, $b$ and $R$, depending smoothly on $\tau$, such that $\sum_l R_l =
		0$ and the horizontal period condition are satisfied.

\end{itemize}

The proofs for these claims are very similar to those for minimal
surfaces~\cite{traizet2002e} with only slight modifications, and they are not
essential for the following analysis of singularities, so we postpone them to
Appendix~\ref{sec:ift}.

\section{Singularities}\label{sec:singularities}

The rest of the paper is devoted to the analysis of singularities.  The study
of minimal surfaces usually avoids singularities, but complete non-planar
maxfaces always appear with singularities.

Recall that the singular set is given by $|g|=1$.  From our definition of the
Gauss map, the singularity set of the maxface $X_t$ is given by the union of
\[
  \mathcal{S}_{l,k} = \{z \in \CC_l : |v_{l,k}|=t\} = \{z \in \CC_{l+1} : |w_{l,k}|=t\}
\]
with $1\leq l\leq N-1, 1 \le k \le n_l$.  We aim to analyze
the nature of these singularities.

For this purpose, we will focus on singularities around a specific neck of
interest, labeled by $(l,k)$.  Without loss of generality, we may assume that
$l$ is odd.  So the Gauss map $g_t = t/v_{l,k} = w_{l,k}/t$ in the local
coordinates.  To ease the text, we will omit the subscript $(l,k)$ unless
necessary.  So we study the connected component $\mathcal{S}$ of the singular
set given by $|v|=|w|=t$.

\subsection{The governing function}\label{sec:Governingfunction}

We need to study the function
\begin{equation}\label{eq:A}
  \cA(t, \theta) = -\frac{g_t(dh_t/dv)}{dg_t/dv}\Big|_{v=e^{i \theta}} = t e^{i \theta} f_t(t e^{i \theta}),
\end{equation}
where $f_t := dh_t/dv$.

Let $p$ be a singular point with $v(p) = t e^{i \theta}$.  On the one hand, it
was proved in~\cite{umehara2006} that the
parameterization~\eqref{eq:maxWeierstrass} is a front (that is, the projection
of a Legendrian immersion into the unit cotangent bundle of $\RR^3$) on a
neighborhood $U$ of a singular point $p$ and $p$ is a nondegenerate singular
point (meaning $d\lambda = 0$ where $\lambda^2$ is the determinant of the
Euclidean metric tensor) if and only if $\re(1/\cA(t, \theta))\ne 0$.  If this
is the case, the singular set $\mathcal{S}$ is a smooth \emph{singular curve}
in $U$ that passes through $p$.  In our case, the singular curve
$\gamma_t(\theta)$ is actually given by $v(\gamma_t(\theta)) = t e^{i \theta}$,
$0 \le \theta < 2\pi$.

On the other hand, it was shown in~\cite{umehara2006} that
\[
 	\det(\dot\gamma_t(\theta), \eta_t(\theta)) = \im(1/\cA(t, \theta)),
\]
where $\dot\gamma_t$ is the \emph{singular direction} and $\eta_t \in
\operatorname{Ker}(dX_t)$ is the \emph{null direction}.  So $\im(1/\cA)$
measures the collinearity between the $\dot\gamma$ and $\eta$.

In particular, $p=te^{i\theta}$ is a cuspidal singularity whenever
\begin{equation} 
  \re \cA(t, \theta) \ne 0 \quad\text{and}\quad  \im \cA(t, \theta) \ne 0,
\end{equation}
and $p$ is a swallowtail singularity whenever
\begin{equation} 
  \re \cA \ne 0, \quad \im \cA = 0, \quad\text{and}\quad \frac{\partial}{\partial\theta} \im \cA \ne 0.
\end{equation}
In fact, the cuspidals are $A_2$ singularities, the swallowtails are $A_3$
singularities and the butterflies are $A_4$ singularities. More generally, one
may define~\cite{honda2021} that $p$ is a \emph{generalized $A_{k+2}$
singularity} if
\begin{equation} \label{eqn:generalAk}
  \re \cA\neq 0, \quad \im \cA = \frac{\partial}{\partial\theta} \im \cA = \cdots \frac{\partial^{k-1}}{\partial\theta^{k-1}} \im \cA = 0, \quad\frac{\partial^k}{\partial\theta^k} \im \cA \ne 0.
\end{equation}
It was proved \cite{izumiya2012, kokubu2005} that generalized $A_k$
singularities with $k=2,3,4$ are indeed $A_k$ singularities, but this is not
known for $k\ge 5$.

\subsection{Derivatives of \texorpdfstring{$\cA$}{A}}\label{sec:Partialderivativeof A}

From equation \ref{eqn:generalAk}, it is clear that we will need higher
derivatives of $\mathcal{A}$ over $t$ to determine the nature of singularities.
The following calculation can be seen as generalizing Traizet's
calculation~\cite{traizet2008t} of the first derivative of $dh$ over $t$.

The height differential $dh_t$ for the maxface, as defined in
Section~\ref{sec:dh}, has the following Laurent expansion in the annulus
$t^2/\varepsilon < v < \varepsilon$:
\[
  dh_t(v) = \sum_{n \in \ZZ} a_n(t) v^n dv, \quad
  dh_t(w) = \sum_{n \in\ZZ} b_n(t) w^n dw,
\]
where
\[
  a_n = \frac{1}{2\pi i} \int_{|v|=\varepsilon}\frac{dh_t}{v^{n+1}}, \quad
  b_n = \frac{1}{2\pi i} \int_{|w|=\varepsilon}\frac{dh_t}{w^{n+1}}.
\]
Moreover, for each $n\in \ZZ$,
\[
  a_n = \frac{1}{2\pi i} \int_{|v|=\varepsilon}\frac{dh_t}{v^{n+1}}
  = \frac{-1}{2\pi i} \int_{|w|=\varepsilon} \frac{dh_t w^{n+1}}{t^{2n+2}}
\]
because of the gluing.  Therefore, we have
\begin{equation}\label{eq:dmandtm}
  \frac{\partial^m a_n}{\partial t^m}
  = \frac{-1}{2\pi i} \sum_{j = 0}^{m}
  \int_{|w|=\varepsilon} w^{n+1} \binom{m}{j} \frac{\partial^{m-j} dh_t}{\partial t^{m-j}}(-2n-2)_j\frac{1}{t^{2n+2+j}}.
\end{equation}
Here, $(a)_j = a(a-1)\cdots(a-j+1)$ (in particular $(a)_0=1)$ is the descending
factorial. In particular, $(-2n-2)_j = 0$ whenever $0 \le -2n-2 < j$. Note that
$dh_t$ and its derivatives are bounded on the circle $|w| = \varepsilon$.

We now prove the following formula for partial derivatives of $\cA$ over $t$.
\begin{equation}\label{eq:dkAdtk}
  \frac{1}{m!}\frac{\partial^m \cA}{\partial t^m}(0,\theta)=\lim_{t\to 0}\sum_{0\leq n\leq m-1}\frac{1}{(m-n-1)!}\frac{\partial^{m-n-1}}{\partial t^{m-n-1}}(a_ne^{i(n+1)\theta}-b_ne^{-i(n+1)\theta}).
\end{equation}

\begin{proof}
 	Recall the expression~\eqref{eq:A} of $\cA$
 	\begin{equation}\label{eq:Ainv}
    \cA(t, \theta) = t e^{i \theta} \frac{dh_t}{dv}(t e^{i \theta})
    = \sum_{n \in \ZZ} a_n(t) t^{n+1} e^{i (n+1) \theta} = \sum_{n \in \ZZ} b_n(t) t^{n+1} e^{-i (n+1) \theta},
 	\end{equation}
 	which can be expanded in terms of $t$ as follows
 	\begin{equation}\label{eq:A(t,theta)}
    \cA(t, \theta) = \sum_{m \geq 0} \frac{1}{m!}\frac{\partial^m \cA}{\partial t^m} (0, \theta)  t^m.
 	\end{equation}
 	From \eqref{eq:Ainv}, we have 
 	\begin{align*}
  	\frac{1}{m!}\frac{\partial^m \cA}{\partial t^m}(t,\theta)
  	&= \sum_{n \in \ZZ} \sum_{j = 0}^m \binom{m}{j} \frac{(n+1)_j}{m!} \frac{\partial^{m-j} a_n}{\partial t^{m-j}} t^{n+1-j} e^{i(n+1)\theta}.
 	\end{align*}
 	In the following, we will calculate the limits of the terms on the right-hand
 	side as $t \to 0$.  The computation is done case by case.

 	\begin{description}
    \item[$n>m-1$] Since derivative of $a_n$ is bounded, $\frac{\partial^{m-j}
    	a_n}{\partial t^{m-j}} t^{n+1-j}\to 0$ as $t\to 0$ when $n>j-1$. In
    	particular, this also holds when $n>m-1$.  We conclude that, when $n>m-1$
    	and as $t\to 0$,
     	\[
        \sum_{j = 0}^m \binom{m}{j} \frac{(n+1)_j}{m!} \frac{\partial^{m-j} a_n}{\partial t^{m-j}} t^{n+1-j} e^{i(n+1)\theta}\to 0.
     	\]

    \item[$n=m-1$] In this case, as $t\to 0$, we have
     	\[
        \sum_{j = 0}^m \binom{m}{j} \frac{(m)_j}{m!} \frac{\partial^{m-j} a_{m-1}}{\partial t^{m-j}} t^{m-j} e^{im\theta}\to a_{m-1}e^{im\theta},
     	\]
     	since derivatives of $a_n$ are bounded and, for $j<m$, $t^{m-j}\to 0.$

    \item[$0 \leq n<m-1$] In this case, as $t \to 0$, we have
     	\[
        \sum_{j = 0}^m \binom{m}{j} \frac{(n+1)_j}{m!} \frac{\partial^{m-j} a_n}{\partial t^{m-j}} t^{n+1-j} e^{i(n+1)\theta}\to \frac{1}{(m-n-1)!}\frac{\partial^{m-n-1}a_n}{\partial t^{m-n-1}}e^{i(n+1)\theta},
     	\]
     	since $j=n+1$ is the only non-zero term in the summation.

    \item[$n=-1$]  In this case, we have
     	\[
      	\sum_{j = 0}^m \binom{m}{j} \frac{(n+1)_j}{m!} \frac{\partial^{m-j} a_n}{\partial t^{m-j}} t^{n+1-j} e^{i(n+1)\theta} = 0.
     	\]

    \item[$-m \le n \le -2$]
     	In this case, we have
     	\begin{align*}
        \frac{\partial^{m-j} a_n}{\partial t^{m-j}} t^{n+1-j}
        &= \frac{-1}{2\pi i}\sum_{k=0}^{m-j} \int_{w = |\varepsilon|} w^{n+1} \binom{m-j}{k} \frac{\partial^{m-j-k} dh_t}{\partial t^{m-j-k}}(-2n-2)_k t^{-n-1-k-j}\\
        &\to \frac{-1}{2\pi i}\sum_{k=-n-1-j}^{m-j} \int_{w = |\varepsilon|} w^{n+1} \binom{m-j}{k} \frac{\partial^{m-j-k} dh_t}{\partial t^{m-j-k}}(-2n-2)_k t^{-n-1-k-j}.
     	\end{align*}
     	By the identity
     	\begin{multline}\label{eq:identity1}
        \sum_{j=0}^l \binom{m}{j} \frac{(n+1)_j}{m!} \binom{m-j}{k}(-2n-2)_k\\
        =\begin{dcases}
          \frac{1}{(m-n-1)!}, & \text{ when } l=-n-1\\
          0, & \text{ when } l<-n-1
        \end{dcases},
     	\end{multline}
     	where $l=j+k$ and $-m \le n \le -2$, we conclude that
     	\[
        \sum_{j=0}^m \binom{m}{j} \frac{(n+1)_j}{m!} \frac{\partial^{m-j} a_n}{\partial t^{m-j}} t^{n+1-j}\to\frac{-1}{(m+n+1)!}\frac{\partial^{m+n+1}}{\partial t^{m+n+1}} b_{-n-2}.
     	\]

    \item[$n=-(m+1)$]  In this case,
     	\begin{align*}
        \frac{\partial^{m-j} a_{-(m+1)}}{\partial t^{m-j}}t^{-m-j}
        &= \frac{-1}{2\pi i} \sum_{k = 0}^{m-j}
        \int_{|w|=\varepsilon} w^{-m} \binom{m-j}{k} \frac{\partial^{m-j-k} dh_t}{\partial t^{m-j-k}}(2m)_k\frac{1}{t^{-m+j+k}}\\ 
        &\to \frac{-1}{2\pi i}\int_{|w|=\epsilon}w^{-m}dh_t(2m)_{m-j}.
     	\end{align*}
     	As $t\to 0$, we have
     	\begin{align*}
        &\sum_{j = 0}^m \binom{m}{j} \frac{(-m)_j}{m!} \frac{\partial^{m-j} a_{-(m+1)}}{\partial t^{m-j}} t^{-m-j} e^{-im\theta}\\
        \to&\frac{-1}{2\pi i}\sum_{j = 0}^m \binom{m}{j} \frac{(-m)_j}{m!}\int_{|w|=\epsilon}w^{-m}dh_t(2m)_{m-j}e^{-im\theta}\\
        =&\sum_{j = 0}^m \binom{m}{j} \frac{(-m)_j}{m!}\frac{(2m)!}{(m+j)!}a_{-(m+1)}t^{-2m}e^{-im\theta}\\
        =&\sum_{j = 0}^m \binom{m}{j} \frac{(-m)_j}{m!}\frac{(2m)!}{(m+j)!}(-1)b_{m-1}e^{-im\theta}.
     	\end{align*}
     	By the identity
     	\begin{equation}\label{eq:identity2}
        \sum_{j=0}^m \binom{m}{j}\frac{(-m)_j(2m)_{m-j}}{m!}=1, \quad m>0,
     	\end{equation}
     	we conclude that, as $t\to 0$,
     	\[
        \sum_{j=0}^m \binom{m}{j} \frac{(-m)_j}{m!} \frac{\partial^{m-j} a_{-m-1}}{\partial t^{m-j}} t^{-m-j}\to -b_{m-1}.
     	\]

    \item[$n<-(m+1)$] In this case, by~\eqref{eq:dmandtm}
     	\[
        \frac{\partial^{m-j} a_n}{\partial t^{m-j}}t^{n+1-j}
        = \frac{-1}{2\pi i} \sum_{k = 0}^{m-j}
        \int_{|w|=\varepsilon} w^{n+1} \binom{m-j}{k} \frac{\partial^{m-j-k} dh_t}{\partial t^{m-j-k}}(-2n-2)_k\frac{1}{t^{n+1+j+k}}. 
     	\]
     	Since partial derivatives of $dh$ are bounded on $|w|=\epsilon$,
     	$\frac{\partial^{m-j} a_n}{\partial t^{m-j}}t^{n+1-j}\to 0$ as $t\to 0$
     	when $n+1+j+k<0$. In particular, this also holds when $n<-(m+1)$. We
     	conclude that, when $n<-(m+1)$ and as $t\to 0$,
     	\[
        \sum_{j = 0}^m \binom{m}{j} \frac{(n+1)_j}{m!} \frac{\partial^{m-j} a_n}{\partial t^{m-j}} t^{n+1-j} e^{i(n+1)\theta}\to 0.
     	\]
 	\end{description}

 	Gathering all these computations, we obtain
 	\begin{align*}
    \frac{1}{m!}\frac{\partial^m \cA}{\partial t^m}(0,\theta)
    &=\lim_{t \to 0}\Bigg(-\sum_{-m-1\leq n\leq -2}\frac{1}{(m+n+1)!}\frac{\partial^{m+n+1}}{\partial t^{m+n+1}} b_{-n-2}e^{i(n+1)\theta}\\
    &\quad\quad\quad\quad+\sum_{0\leq n\leq m-1}\frac{1}{(m-n-1)!}\frac{\partial^{m-n-1}a_n}{\partial t^{m-n-1}}e^{i(n+1)\theta}\Bigg)\\
    &=\lim_{t\to 0}\sum_{0\leq n\leq m-1}\frac{1}{(m-n-1)!}\frac{\partial^{m-n-1}}{\partial t^{m-n-1}}(a_ne^{i(n+1)\theta}-b_ne^{-i(n+1)\theta}).
 	\end{align*}
\end{proof}

The proof above involves two combinatorial identities, namely
Eqs~\eqref{eq:identity1} and~\eqref{eq:identity2}, which we now prove.

\begin{proof}[Proof of \eqref{eq:identity2}]
 	\begin{align*}
  	\sum_{j=0}^m \binom{m}{j}\frac{(-m)_j(2m)_{m-j}}{m!} &= \sum_{j=0}^m (-1)^j \frac{m!(m+j-1)!(2m)!}{m!(m-j)!j!(m-1)!(m+j)!}\\
  	&=\sum_{j=0}^m (-1)^j \frac{m}{m+j} \binom{m+j}{j} \binom{2m}{m+j}\\
  	&=\sum_{j=0}^m (-1)^j \frac{m}{m+j} \binom{m}{j} \binom{2m}{m}\\
  	&=m \binom{2m}{m} \sum_{j=0}^m \frac{(-1)^j}{m+j} \binom{m}{j} = 1,
 	\end{align*}
 	where the last line follows from
 	\begin{align*}
  	\sum_{j=0}^m \frac{(-1)^j}{m+j} \binom{m}{j} &=
  	\sum_{j=0}^m (-1)^j \binom{m}{j} \int_0^1 x^{m+j-1} dx\\
  	&=\int_0^1 x^{m-1} \sum_{j=0}^m \binom{m}{j} (-x)^j dx=\int_0^1 x^{m-1}(1-x)^m dx \\
  	&=\frac{\Gamma(m)\Gamma(1+m)}{\Gamma(1+2m)}=\frac{(m-1)!m!}{(2m)!}=\frac{1}{m\binom{2m}{m}}.
 	\end{align*}
\end{proof}

\begin{proof}[Proof of \eqref{eq:identity1}]
 	For each $-n-1 \le l= j+k \le m$,
 	\begin{align*}
  	\sum_{j=0}^l \binom{m}{j} \frac{(n+1)_j}{m!} \binom{m-j}{k}(-2n-2)_k 
  	&= \sum_{j=0}^l \binom{m}{l} \binom{l}{j}\frac{(n+1)_j}{m!} (-2n-2)_{l-j}\\
  	&=\binom{m}{l}\frac{l!}{m!} \sum_{j=0}^l \binom{l}{j}\frac{(n+1)_j(-2n-2)_{l-j}}{l!}
 	\end{align*}
 	which, by similar argument as before, equals $\frac{1}{(m+n+1)!}$ when
 	$l=-n-1$.  Otherwise, if $l > -n-1$, it
 	\begin{align*}
  	&=\binom{m}{l}\frac{l!}{m!} \sum_{j=0}^l (-1)^j \frac{l!(-n+j-2)!(-2n-2)!}{l!(l-j)!j!(-n-2)!(-2n-2-l+j)!}\\
  	&=\binom{m}{l}\frac{l!}{m!} \sum_{j=0}^l (-1)^j \binom{-n+j-2}{j} \binom{-2n-2}{-2n-2-l+j}\\
  	&=\binom{m}{l}\frac{l!}{m!} \sum_{j=0}^l (-1)^j \frac{(-n-2+j)_{l+n}}{(-n-2)_{l+n}} \binom{-2n-2-l+j}{j} \binom{-2n-2}{-2n-2-l+j}\\
 	&=\binom{m}{l}\frac{l!}{m!} \sum_{j=0}^l (-1)^j \frac{(-n-2+j)_{l+n}}{(-n-2)_{l+n}} \binom{l}{j} \binom{-2n-2}{-2n-2-l}=0\end{align*}
 	because $(-n-2+j)_{l+n}$ is a polynomial of $j$ of degree $0 \le l+n < l$.
\end{proof}

\subsection{Nature of singularities}\label{sec:natureofSIngularties}

We are now in the possition to analyze the nature of singularities on the
maxfaces we constructed by node-opening. 

Recall that $dh_t$ extend real analytically to $t=0$ with the form given
by~\eqref{eq:heightdiff}, with simple poles of residue $-r_{l,k}$ at the nodes
at $a_{l,k}$.  Because $-r_{l,k}\to -c_l$ as $t \to 0$, we have $\cA(t,\theta)
\to  -c_l \neq 0$ no matter the value of $\theta$.  This implies that
$A(t,\theta)$ extends real analytically to $t=0$ with a non-zero finite value
independent of $\theta$.  So $\re(1/\cA(t,\theta))$ extends real analytically
to $t=0$ with a non-zero value independent of $\theta$.  By continuity, we have
$\re(1/\cA) \ne 0$ for sufficiently small $t$.

This proves that, for sufficiently small non-zero $t$, the Weierstrass
parameterization defines a front in a neighborhood of the singular points, and
the singular points are all nondegenerate.  Note that, for sufficiently small
$t$, the singular set is a circle of radius $t$ in the local coordinate $v$,
which obviously defines a smooth curve.  So, the nondegeneracy of the singular
points is expected.

\subsubsection{Non-cuspidal singularities}\label{subsec:noncuspidal}

The node opening could also be implemented by an identification $v w = s$ where
$s$ is a complex parameter. It was proved in~\cite{traizet2002e} that
the height differential $dh$ depends holomorphically on $s$ and $v$, and
extends holomorphically to $s=0$.  In our case, we have $v f_t$ extends
holomorphically to $(s, v)=(0, 0)$ with the value $-c_l$.  So $\cA(t, \theta)$
depends real analytically on $t$ and $\theta$, and extends real analytically to
$t=0$ with the value $-c_l$ independent of the value of $\theta$.

The singularity is cuspidal when $\im \cA \ne 0$.  So, the set of non-cuspidal
singularities around a neck, given as the zero locus $\im \cA = 0$, is a real
analytic variety.

If the zero locus has a non-zero measure, then $\im \cA \equiv 0$, and the
singular curve is mapped to a single point, so we have a cone singularity no
matter $t$ and $\theta$~\cite{fujimori2009}.

Otherwise, by Lojasiewicz's theorem, the non-cuspidal singular set can be
stratified into a disjoint union of real analytic curves ($1$-strata) and
discrete points ($0$-strata).  In particular, $t=0$ is a trivial solution of
$\im \cA=0$, and there is no $0$-strata for $t \ne 0$ sufficiently small.  In
other words, in a neighborhood of $t=0$, the set of non-cuspidal singularities
is given by disjoint curves.  See Figure~\ref{fig:variety}.

\begin{figure}[h!]
  \includegraphics[width=\textwidth]{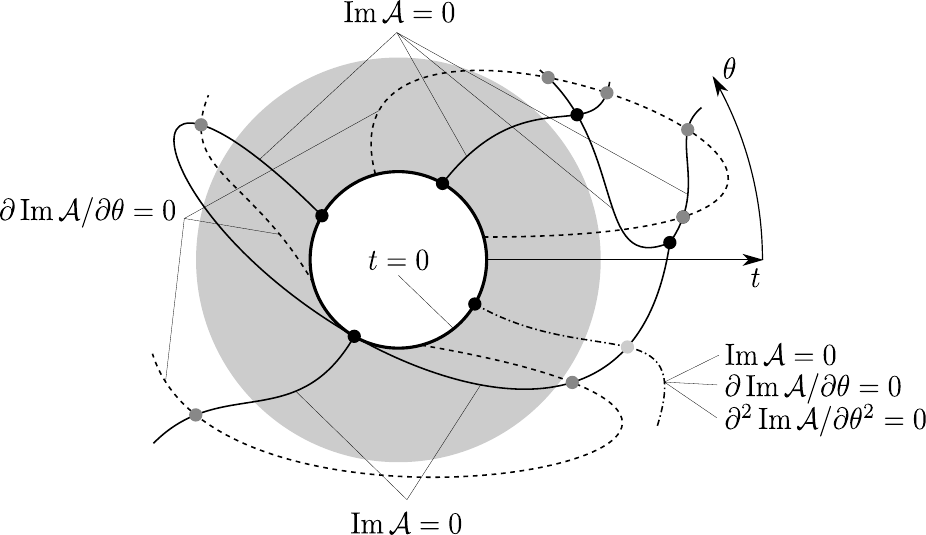}
  \caption{
    Sketch of a typical structure of non-cuspidal singularities. The circle in
    the middle is the trivial locus with $t=0$.  Solid curves are solutions
    to~$\im\cA = 0$.  Dashed curves are solutions to $\partial \im\cA /
    \partial \theta = 0$.  Dots indicate the $0$-strata.  The grey dots at the
    intersection of solid and dashed curves are then at least butterfly
    singularities.  The dot-dashed curve indicates a possible curve that solves
    $\im\cA = \partial \im\cA / \partial \theta = \partial^2 \im\cA / \partial
    \theta^2 = 0$. Singularities on this curve are then at least generalized
    $A_5$ singularities.  The grey area indicates a neighborhood of the trivial
    locus that includes no $0$-strata of the variety. The variety appears as
    disjoint curves within this neighborhood.  \label{fig:variety}
  }
\end{figure}

More generally, the set of generalized $A_k$-singularities, $k > 0$, is a
real-analytic variety given by the zero locus
\[
 	\mathcal{Z}_k = \{(t, \theta) : \im\cA = \frac{\partial}{\partial\theta} \im \cA = \cdots \frac{\partial^{k-1}}{\partial\theta^{k-1}} \im \cA = 0\}.
\]
Again, by Lojasiewicz's theorem, the locus can be stratified into curves and
discrete points and contains the trivial solution $t=0$.  In particular, if the
singularities are of type $A_k$ along a segment of a curve in $\mathcal{Z}_k$,
then the type will remain along this curve until hitting a $0$-stratum of
$\mathcal{Z}_{k+1}$. See Figure~\ref{fig:variety}.

We have proved the following

\begin{theorem}\label{thm:noncuspidalvariety}
  For $t \ne 0$ sufficiently small, if the height differential $dh_t$ is not
  identically $0$ (as a function of $t$ and $\theta$), then the non-cuspidal
  singularities around a neck are described by a disjoint union of finitely
  many curves in the $(t, \theta)$-plane, each given by a real-analytic
  function $\theta = \theta(t)$.  Moreover, along each of these curves, the
  type of singularities is invariant for $t \ne 0$ sufficiently small.
\end{theorem}

\subsubsection{Swallowtails}\label{subsec:swallowtails}

We have seen that, generically, a non-cuspidal singularity is a swallowtail.
In this part, for sufficiently small non-zero $t$, we want to identify
swallowtails using the Implicit Function Theorem.  The strategy is the
following:

We first remove the trivial solutions $t=0$ by considering the function
\begin{equation}\label{eq:Atilde}
 	\widetilde\cA(t, \theta) := \im\cA(t, \theta)/t^m,
\end{equation}
which should extend to $t=0$ with the values $\widetilde\cA(0, \theta)$ that is
not identically $0$.  Of course, this is only possible if $\im\cA(t, \theta)$
itself is not identically $0$.  That is if the singularity is not cone-like.
Then $\widetilde\cA(0, \theta)$ could only have finitely many zeros.  At a
simple zero $\theta_0$, we may apply the Implicit Function Theorem on
$\widetilde\cA$.  More specifically, if
\[
 	\widetilde\cA(0, \theta_0)=0, \qquad \frac{\partial}{\partial\theta}\widetilde\cA(0, \theta_0) \ne 0
\]
for some $\theta_0$, then for sufficiently small $t$, there exists a unique
value for $\theta$ as a function of $t$, such that $\widetilde\cA(t, \theta(t))
= 0$ and $\theta(t)$ extends to $t=0$ with the value $\theta(0) = \theta_0$.
Moreover, $\frac{\partial}{\partial\theta} \widetilde\cA(t, \theta(t)) \ne 0$
for sufficiently small non-zero $t$.  In other words, the singularities are
swallowtails along the curve $\theta=\theta(t)$.  Unfortunately, if $\theta_0$
is a multiple zero of $\widetilde\cA$, we are not able to draw concrete
conclusions on the numbers and types of the singularities in the neighborhood
$(0, \theta_0)$.

Now, the problem reduces to finding $\widetilde\cA$.  If
$\frac{\partial^k}{\partial t^k} \im \cA(0, \theta) \equiv 0$ for all $0 \le k
< m$, then $\frac{\partial^m}{\partial t^m} \im \cA(t, \theta)$ extends to
$t=0$ with the value
\[
  \frac{\partial^m}{\partial t^m} \im \cA(0, \theta) = m! \lim_{t \to 0} \frac{\im\cA(t, \theta)}{t^m}.
\]
Therefore, let $m$ be the smallest integer $k$ such that
$\frac{\partial^k}{\partial t^k} \cA(0, \theta) \not\equiv 0$, then
\[
  \widetilde\cA(t, \theta) = \frac{\im\cA(t,\theta)}{t^m} = \frac{1}{m!} \frac{\partial^m}{\partial t^m} \im\cA(t, \theta) + o(1)
\]
for $t$ in a neighborhood of $0$, and extends to $t=0$ with the value
$\frac{1}{m!}\frac{\partial^m}{\partial t^m} \im\cA(0, \theta)$.

In Section \ref{sec:Partialderivativeof A}, we proved~\eqref{eq:dkAdtk} for
the partial derivatives, which we repeat below
\[
  \frac{1}{m!}\frac{\partial^m \cA}{\partial t^m}(0,\theta)=
  \lim_{t\to 0}\sum_{0\leq n\leq m-1}\frac{1}{(m-n-1)!}\frac{\partial^{m-n-1}}{\partial t^{m-n-1}}(a_ne^{i(n+1)\theta}-b_ne^{-i(n+1)\theta}).
\]
where
\[
  a_n = \frac{1}{2\pi i} \int_{|v|=\varepsilon}\frac{dh_t}{v^{n+1}}, \quad
  b_n = \frac{1}{2\pi i} \int_{|w|=\varepsilon}\frac{dh_t}{w^{n+1}}
\]
are coefficients in the Laurent expansions of $dh_t$ in the $\CC_l$ and
$\CC_{l+1}$, respectively.  Recall that, at $t=0$, $dh_t/dv = g_l = 1/v$ on
$\CC_l$, so
\[
  a_n = \operatorname{Res}_0 g_l^{n+2}, \quad b_n = \operatorname{Res}_0 g_{l+1}^{n+2}.
\]
In particular,
\[
 	\frac{\partial \cA}{\partial t}(0,\theta) = \lim_{t\to 0} (a_0 e^{i\theta} - b_0 e^{-i\theta}).
\]
Note that, at $t=0$, we have
\begin{align*}
 	\overline{a_0} &= \sum_{1 \le j \ne k \le n_l}\frac{c_l}{p_{l,k}-p_{l,j}}-\sum_{1 \le j \le n_{l-1}}\frac{c_{l-1}}{p_{l,k}-p_{l-1,j}},\\
 	b_0 &= \sum_{1 \le j \ne k \le n_l}\frac{c_l}{p_{l,k}-p_{l,j}}-\sum_{1 \le j \le n_{l+1}}\frac{c_{l+1}}{p_{l,k}-p_{l+1,j}},
\end{align*}
so $\overline{a_0} + b_0 = F_{l,k}/c_l$, which vanishes by the balance
condition.  Therefore, $\frac{\partial}{\partial t}\cA(0,\theta) = 2 \re(a_0
e^{i \theta})$ is real, hence $\frac{\partial}{\partial t}\im\cA(0,\theta)=0$.
This implies that $m>1$ in~\eqref{eq:Atilde}.

Then, we must look at the next derivative, namely
\[
 	\frac{1}{2}\frac{\partial^2 \cA}{\partial t^2}(0,\theta) = \lim_{t\to 0} (a_1 e^{2i\theta}-b_1 e^{-2i\theta}) + \lim_{t \to 0} \frac{\partial}{\partial t} (a_0 e^{i\theta}-b_0 e^{-i\theta}).
\]
Note that the node opening process remains the same if we replace $t$ by $-t$,
so $dh_t$, as well as its Laurent coefficients, are even in $t$.  So, the
second limit vanishes in the formula above.  The imaginary part of the first
limit equals $R^{(1)}(\theta)$ as defined in Section~\ref{sec:mainresult}.  If
it does not vanish, then $m=2$ in~\eqref{eq:Atilde}, and
$\widetilde\cA(0,\theta)$ is given by a shifted sine function of period $\pi$.
We then conclude that
\begin{proposition}\label{prop:swallowtails}
  If $R^{(1)}(\theta) \not\equiv 0$ at $t=0$, there are four non-cuspidal
  singularities around the neck for sufficiently small non-zero $t$, they are
  all swallowtails and, as $t \to 0$, the differences between the angles
  $\theta$ of neighboring swallowtails tend to $\pi/2$.  In other words, these
  swallowtails tend to be evenly distributed as $t \to 0$.
\end{proposition}

Otherwise, if $R^{(1)}(\theta) \equiv 0$ at $t=0$, we must look at higher order
derivatives of $\cA$ and continue the analysis.  But things become
significantly more complicated, mainly because we don't have control over
even-order derivatives of the $a_n$ and $b_n$.

\subsubsection{Symmetries}\label{sec:symmetries}

We can say more about the singularities if symmetries are imposed to the maxfaces. 

\begin{proposition}\label{prop:rotations}
  Assume that the configuration has a rotational symmetry of order $r>1$ and
  the neck of interest is at the rotation center.  If $R^{(r-1)}(\theta)
  \not\equiv 0$, then there are $2r$ non-cuspidal singularities around the neck
  for sufficiently small non-zero $t$, they are all swallowtails and, as $t \to
  0$, the differences between the angles $\theta$ between neighboring
  swallowtails tend to $\pi/r$.  In other words, these swallowtails tend to be
  evenly distributed as $t \to 0$.
\end{proposition}

\begin{proof}
  Under the assumed symmetry, $vf=vdh_t/dv$ is a function of $v^r$, hence $a_0
  = b_1 = \cdots = a_{r-2} = b_{r-2} = 0$ for all $t$.  then $m \ge r-1$
  in~\eqref{eq:Atilde} and the equality holds if $R^{(r-1)}(\theta) \not\equiv
  0$ at $t=0$.  In the case of equality, $\widetilde\cA(0,\theta)$ is given by
  a shifted sine function of period $2\pi/r$.
\end{proof}

\begin{proposition}\label{prop:verticalref}
  Assume that the configuration has a vertical reflection plane that cuts
  through the neck of interest.  Then, the singularity around the neck that is
  fixed by the reflection is non-cuspidal.
\end{proposition}

\begin{proof}
  We may further assume that the singular point $p$ fixed by the reflection is
  given by $v(p) = t$.  Then, the height differential is real on the real line
  under the local coordinate $v$.  So, all the Laurent coefficients are real no
  matter the value of $t$.  In particular, $\im\cA(t, 0)=0$, so $p$ is a
  non-cuspidal singularity.
\end{proof}

\begin{remark}
  For sufficiently small non-zero $t$, a singularity around the neck that is
  fixed by a vertical reflection could be a generalized $A_k$ singularity only
  for odd $k$.
\end{remark}

\begin{remark}
  By the two propositions above, if a configuration has a dihedral symmetry of
  order $2r$ and the neck of interest is at the symmetry center, then there are
  $2r$ swallowtails around the neck with the same dihedral symmetry.
\end{remark}

\begin{proposition}\label{prop:horizontalref}
  Assume that the configuration of necks has a horizontal reflection plane that
  cuts through the neck of interest.  Then, the singular curve around the neck
  is mapped to a conelike singularity.
\end{proposition}

\begin{proof}
  Under the assumed symmetry, the singular curve is pointwise fixed by an
  antiholomorphic involution $\iota : v \mapsto w = t^2/v$ of the Riemann
  surface $\Sigma_t$, and $\iota^*(dh) = -\overline{dh}$.  In other words, we
  have $a_n + \overline{b_n} = 0$ for all $n \in \ZZ$ no matter the
  value of $t$.  As a consequence, the partial derivatives of $\cA$ over $t$
  are all real, so $\im\cA \equiv 0$.
\end{proof}

\appendix

\section{Using the Implicit Function Theorem} \label{sec:ift}

In Section \ref{sec:WeierstrassData}, we have proposed the Weierstrass data,
namely a Riemann surface $\Sigma$,  Gauss map $g$, and a holomorphic 1- form
$dh$, with undetermined parameters $\alpha_{l,k}$, $\beta_{l,k}$, $a_{l,k},
b_{l,k}, r_{l,k}, R_l,$ and $t$.  In this section will prove that for
sufficiently small $t$, we can find values for all the parameters, as smooth
functions of $t$, such that the triplet $(\Sigma, g, dh)$ becomes a Weierstrass
data for a maxface.   As the argument is similar to that of
\cite{traizet2002e}, we will only provide a sketch and highlight the necessary
changes.

We want to find parameters $X=(t, a, b, \alpha, \beta, r, R)$ that solve the
divisor conditions, period conditions, and regularity conditions.  All
parameters vary in a neighborhood of their initial values at $t=0$, denoted by
$X^\circ$.  Given a balanced configuration $(l, p, c)$, we will see that

\begin{gather*}
 	t^\circ = 0, \quad R_l^\circ = Q_l,\\
 	\forall 1 \le k \le n_l, r_{l,k}^\circ = -\alpha_{l,k}^\circ = \beta_{l,k}^\circ = c_l\\
 	\forall 1 \le k \le n_l, a_{l,k}^\circ = \overline{b_{l,k}^\circ} = \begin{cases}
  	\overline{p_{l,k}}, & l \text{ odd}\\
  	p_{l,k}, & l \text{ even}.
 	\end{cases}
\end{gather*}
The argument in \cite{traizet2002e} applies, word by word, to prove the
following

\begin{proposition}[Divisor condition]
  For $(t, a, b, r, R)$ in a neighborhood of their initial values, there exist
  unique values for $\alpha$ and $\beta$, depending analytically on $(t, a, b,
  r, R)$, such that the divisor conditions are satisfied.  Moreover, at $t=0$,
  we have $-\alpha_{l,k} = \beta_{l,k} = r_{l,k}$. 
\end{proposition}

For $1 \le l < L$, $1 < k \le n_l$, let $\Gamma_{l,k}$ be a closed curve that
starts in $\Omega_l$, travels first through the neck $(l,1)$ to $\Omega_{l+1}$,
then through the neck $(l, k)$ back to $\Omega_l$, and finally close itself.
See~\cite{traizet2002e} for formal definitions of these curves.  For $1
\le l < L$ and $1 < k \le n_l$, the curves $\Gamma_{l,k}$ and the previously
defined $\gamma_{l,k}$ form a homology basis.  So, we only need to close
periods on these curves to solve the period conditions.

Recall that the vertical periods are already closed when defining the height
differential $dh$.  In the following proposition, we need to switch to the
parameter $\tau$ given by $t = \exp(-1/\tau^2)$.  Again, The argument in
\cite{traizet2002e} applies word by word.  The key point is that
$-\tau^{-2}\int_{\Gamma_{l,k}} dh$ extends to a smooth function at $\tau =0$
with the value $2(r_{l,k} - r_{l,1})$.

\begin{proposition}[Vertical periods]
  Assume that $(\alpha, \beta)$ are given by the previous proposition.  For
  $(\tau, a, b, R)$ in a neighborhood of their initial values, there exists
  unique values for $r$, depending smoothly on $(\tau, a, b, R)$, such that the
  vertical period condition~\eqref{eq:vperiod} are satisfied over the curves
  $\Gamma_{l,k}$, $1 \le l < L$ and $1 < k \le n_l$.  Moreover, at $\tau=0$, we
  have $r_{l,k} = c_l$ for all $1 \le k \le n_l$ and $1 \le l < L$, where $c_l$
  are defined from $R_l$ by $c_0 = c_L = 0$ and $c_{l-1} n_{l-1} - c_l n_l =
  R_l$. 
\end{proposition}

The proof for the following step differs from minimal
surfaces~\cite{traizet2002e} only by a few signs.  This slight
difference comes from the sign change in the horizontal period
condition~\eqref{eq:hperiod}.  We will give a sketch to point out the
difference.

\begin{proposition}[Horizontal periods]
  Given a balanced and rigid configuration $(p, Q)$ such that the map $Q \to W$
  has rank $1$.  Assume that $(\alpha, \beta, r)$ are given by previous
  propositions.  For $\tau$ in a neighborhood of $0$, there exists unique
  values for $a$, $b$, and $R$, depending smoothly on $\tau$, such that $\sum_l
  R_l = 0$ and the horizontal period condition~\eqref{eq:hperiod} are satisfied
  over the curves $\Gamma_{l,k}$ and $\gamma_{l,k}$, $1 \le l < L$ and $1 < k
  \le n_l$.  Moreover, at $\tau=0$, up to a translation in $\CC_l$, we
  have $a_{l,k} = \overline{b_{l,k}} = \overline{p_{l,k}}$ if $l$ is odd, $=
  p_{l,k}$ if $l$ is even, and $R_l = Q_l$. 
\end{proposition}

\begin{proof}[Sketched proof]
  Define the horizontal period along a curve $c$ as
  \[
    P(c) = \overline{\int_c g^{-1} dh} + \int_c g dh.
  \]

  Then $tP(\Gamma_{l,k})$ is extends to a smooth function at $\tau=0$ with the
  values
  \[
    \begin{cases}
      b_{l,k} - b_{l,1} + \overline{a_{l,1}} - \overline{a_{l,k}}, & l \text{ odd};\\
      \overline{b_{l,k}} - \overline{b_{l,1}} + a_{l,1} - a_{l,k}, & l \text{ even}.
    \end{cases}
  \]
  If we normalize $b$ by fixing $b_{l,1} = \overline{a_{l,1}}$, then
  $tP(\Gamma_{l,k})$ vanish at $\tau = 0$ if $b_{l,k} = \overline{a_{l,k}}$.
  As the partial differential of $tP(\Gamma_{l,k})$ with respect to $b$ is a
  linear isomorphism, the parameters $b$ are found by the Implicit Function
  Theorem.

  Using these values of $b$, $t^{-1} P(\gamma_{l,k})$ extends to a smooth
  function at $\tau=0$ with the values
  \[
    \begin{dcases}
      - 4\pi i \Bigg(
        \sum_{1 \le j \ne k \le n_l}\frac{2c_l^2}{a_{l,k} - a_{l,j}}
        -\sum_{j=1}^{n_{l+1}} \frac{c_l c_{l+1}}{a_{l,k}-\overline{a_{l+1,j}}}
        -\sum_{j=1}^{n_{l-1}} \frac{c_l c_{l-1}}{a_{l,k}-\overline{a_{l-1,j}}}
      \Bigg),& l \text{ odd},\\
      4\pi i \Bigg(
        \sum_{1 \le j \ne k \le n_l}\frac{2c_l^2}{\overline{a_{l,k}} - \overline{a_{l,j}}}
        -\sum_{j=1}^{n_{l+1}} \frac{c_l c_{l+1}}{\overline{a_{l,k}}-a_{l+1,j}}
        -\sum_{j=1}^{n_{l-1}} \frac{c_l c_{l-1}}{\overline{a_{l,k}}-a_{l-1,j}}
      \Bigg),& l \text{ even}.
    \end{dcases}
  \]
  They vanish at $t = 0$ if $a_{l,k} = \overline{p_{l,k}}$ where $p$ is from a
  balanced configuration.  Since the configuration is rigid, we may
  re-normalize $a$ by fixing two of the $a$ parameters, then use the Implicit
  Function Theorem to find the remaining $N-2$ $a$ parameters, depending
  smoothly on $t$ that solves $t^{-1} P(\gamma_{l,k}) = 0$ for all but two
  necks.

  It remains to solve $P(\gamma_{l,k}) = 0$ for the remaining two necks.  It is
  necessary that $n_{l_0} > 1$ for some $l_0$; otherwise, the configuration
  would not be balanced unless $N=1$.  So we may assume that the remaining
  necks are labeled by $(l_0, 1)$ and $(l_0,k_0)$, $1 < k_0 \le n_{l_0}$.  The
  relation that $P(\gamma_{l_0,1}) + P(\gamma_{l_0,k_0})=0$ follows from the
  Residue Theorem.  The Riemann Bilinear Relation shows that
  \[
    \re \Big( P(\gamma_{l_0,k_0}) \int_{\Gamma_{l_0,k_0}} g^{-1} dh \Big) = 0.
  \]
  And finally, we study the function
  \[
    G = \im\Big(\sum_{l = 1}^L \sum_{k = 1}^{n_k} (-1)^k \overline{p_{l,k}}t^{-1} P(\gamma_{l,k})\Big).
  \]
  It extends to a smooth function at $\tau = 0$ with the values of $4\pi W$,
  which vanishes because the configuration is balanced.  Since the partial
  differential of $W$ with respect to $R$ is surjective at $t=0$, we may use
  the Implicit Function Theorem to find $R$, depending smoothly on $\tau$ in a
  neighborhood of $0$, such that $\sum R_l = 0$ and $G(\tau, R) = 0$. These
  conclude the proof that $P(\gamma_{l_0,1}) = P(\gamma_{l_0,k_0})=0$.
\end{proof}

We have constructed a family of maximal maps. 
\[
 	(X_1, X_2, X_3): \Sigma_t \to \EE^3_1.
\]
Let $0_k$ be the origin point of $\CC_k$.  With a translation if
necessary, we may assume that $0_k \in \Omega_k$.  With similar computations as
in~\cite{traizet2002e}, one verifies that

\begin{itemize}
 	\item The necks converge to Lorentzian catenoids and, after a scaling by $t$, the
 		limit positions of the necks are $p_{l,k}$.

 	\item The image of $\Omega_k$ is a space-like graph over the horizontal plane
 		and this image stays within a bounded distance from $X_3(0_k) + R_k
 		\log(1+t|x_1 + i x_2|)$.

 	\item $X_3(0_k) - X_3(0_{k+1}) = O(-\log t)$. So, if $Q_k < Q_{k+1}$, then
  	for sufficiently small $t$, we have $R_k \le R_{k+1}$ and the image of
  	$\Omega_k$ is above the image of $\Omega_{k+1}$.
\end{itemize}

The singular set, given by $|v_{l,k}|=t$ and $|w_{l,k}|=t$, is compact in
$\Sigma_t$, and is not included in $\Omega_k$.  We then have proved that the
constructed maximal maps are in fact maxface. Moreover these  are embedded in a
wider sense for sufficiently small $t$ if $Q_1 < Q_2 < \cdots < Q_{L}$.

\bibliography{ref.bib}

\end{document}